\documentclass[authoryear,preprint,review,12pt]{elsarticle}

\usepackage{amsmath,amsbsy,amssymb,amsthm}
\usepackage{multirow,multicol,tabularx,booktabs}
\usepackage{graphicx,placeins,color,url}
\usepackage{bbm,bm}
\usepackage[ruled]{algorithm2e}
\usepackage[shortlabels]{enumitem}
\usepackage{mathrsfs}
\usepackage{textcase}

\usepackage[top=2.5cm,bottom=2.5cm, right=2.5cm,left=2.5cm]{geometry}
\usepackage[hidelinks]{hyperref}

\newtheorem{conj}{Conjecture}[section]
\newtheorem{thm}[conj]{\bf Theorem}

\newtheorem{lemma}[conj]{\bf Lemma}

\def\to{\rightarrow}

\def\Bias{\operatorname{Bias}}

\def\Var{\operatorname{Var}}

\def\Med{\operatorname{Med}}

\def\om{\omega}

\def\Cc{{\mathscr C}}

\def\Fc{{\mathscr F}}

\def\Rb{{\mathbb R}}

\def\x{ {\bf x}}

\def\1{\mathbbm{1}}
\def\floorbeta{{\lfloor \beta \rfloor}}

\def\wh{\widehat}

\journal{Statistics \& Probability Letters}

\begin{document}

\begin{frontmatter}



\title{Lower bounds for the trade-off between bias and mean absolute deviation}

\author[label1]{Alexis Derumigny}
\author[label2]{Johannes Schmidt-Hieber}
\affiliation[label1]{organization={TU Delft},
             addressline={Mekelweg 5},
             city={Delft},
             postcode={2628 CD},
             country={The Netherlands}}
\affiliation[label2]{organization={University of Twente},
             addressline={Drienerlolaan 5},
             city={Enschede},
             postcode={7522 NB},
             country={The Netherlands}}



\begin{abstract}
In nonparametric statistics, rate-optimal estimators typically balance bias and stochastic error. The recent work on overparametrization raises the question whether rate-optimal estimators exist that do not obey this trade-off.
In this work we consider pointwise estimation in the Gaussian white noise model with regression function $f$ in a class of $\beta$-H\"older smooth functions. Let 'worst-case' refer to the supremum over all functions $f$ in the H\"older class. It is shown that any estimator with worst-case bias $\lesssim n^{-\beta/(2\beta+1)}=: \psi_n$ must necessarily also have a worst-case mean absolute deviation that is lower bounded by $\gtrsim \psi_n.$ To derive the result, we establish abstract inequalities relating the change of expectation for two probability measures to the mean absolute deviation.
\end{abstract}




\end{frontmatter}



\section{Motivation}

Inspired by recent claims that overparametrization challenges the traditional view on the bias-variance trade-off, see for instance 
\cite{belkin2018reconciling, neal2018modern, 2019arXiv191208286N}, we aim to quantify the extend to which the trade-off between bias and stochastic error in nonparametric and highdimensional statistics is universal. The recent work \cite{2020arXiv200600278D} derives lower bounds for the bias-variance trade-off covering standard nonparametric and high-dimensional statistical models. In this work, we take this one step further by deriving a universal lower bound for the trade-off between bias and mean absolute deviation. Such universal lower bounds immediately translate into universal lower bounds for the bias-variance trade-off and are thus stronger. 

Another motivation for our work is that for constructing confidence bands with small diameter in function estimation problems, one needs to find an upper bound for the bias. The bias is hard to estimate from data, see e.g. \cite{MR3127852}. To obtain a small confidence bands, it is therefore desirable to find rate-optimal estimators with negligible bias. Universal lower bounds on the trade-off between bias and stochastic error can be a tool to show that this is impossible in the sense that decreasing the bias necessarily increases the stochastic error.  

\section{Summary of previous work on universal lower bounds for the bias-variance trade-off}

The previous work \cite{2020arXiv200600278D} derives universal lower bounds for the bias-variance trade-off. For nonparametric function estimation, evaluating an estimator either via the squared pointwise risk or the mean integrated squared error,
it is shown that there exists a universal bias-variance trade-off that also cannot be overcome by fitting overparametrized models.

For estimation of a high-dimensional sparse vector in the Gaussian sequence model, the situation is different and the bias-variance trade-off does not always hold. \cite{2020arXiv200600278D} shows that there are estimation problems driven by the worst-case bias. While the convergence rate of the worst-case variance cannot be arbitrarily fast, it can be considerably faster than the minimax estimation rate. 

These lower bounds on the bias-variance trade-off rely on a number of abstract inequalities that all relate the variance to the changes that occur if expectations are taken with respect to different probability measures. To outline the idea, we recall one of these change of expectation inequalities that we modify later on. Let $P$ and $Q$ be two probability distributions on the same measurable space. Denote by $E_P$ and $\Var_P$ the expectation and variance with respect to $P$ and let $E_Q$ and $\Var_Q$ be the expectation and variance with respect to $Q.$ The squared Hellinger distance is defined by $H(P,Q)^2 := \tfrac12 \int (\sqrt{p(\om)}-\sqrt{q (\om)})^2 \, d\nu(\om)$ with $\nu$ a measure dominating both $P$ and $Q$ and $p,q$ the respective $\nu$-densities of $P$ and $Q.$ It can be checked that the Hellinger distance does not depend on the choice of $\nu.$

\begin{lemma}[Lemma 2.1 in \cite{2020arXiv200600278D}]\label{lem.general_lb}
For any random variable $X,$
\begin{align}
    \frac{( E_P[X]-E_Q[X])^2}{4-2H^2(P,Q)}
    \Big( \frac{1}{H(P,Q)}-H(P,Q)\Big)^2
    &\leq \Var_P(X)+ \Var_Q(X).
    \label{eq.lem1_2}
\end{align}
\end{lemma}

To derive a lower bound on the bias-variance trade-off from such an inequality, consider a statistical model $(P_\theta:\theta\in \Theta)$ with $P_\theta$ the distribution of the data for parameter $\theta$ and $\Theta \subseteq \mathbb{R}$ the parameter space. Choosing two parameters $\theta,\theta'\in \Theta,$ inequality \eqref{eq.lem1_2} shows that for any estimator $\wh \theta,$ 
\begin{align}
        \frac{( E_\theta[\wh \theta]-E_{\theta'}[\wh \theta])^2}{4-2r^2(\theta,\theta')}
    \Big( \frac{1}{r(\theta,\theta')}-r(\theta,\theta')\Big)^2
    &\leq \Var_\theta\big(\wh \theta\big)+ \Var_{\theta'}\big(\wh \theta\big),
    \label{eq.1}
\end{align}
with $r(\theta,\theta'):=H({P_\theta},P_{\theta'}),$ $E_\theta:=E_{P_\theta}$, and $\Var_\theta:=\Var_{P_\theta}.$ Introducing the bias $\Bias_\theta(\wh \theta)=\theta-E_\theta[\wh \theta],$ one can now rewrite the difference of the expectations as $E_\theta[\wh \theta]-E_{\theta'}[\wh \theta]=\theta-\theta'-\Bias_\theta(\wh \theta)+\Bias_{\theta'}(\wh \theta).$ If we assume that the bias is smaller than some value, say $B$, and take the parameters $\theta,\theta'$ sufficiently far apart, such that $|\theta-\theta'|\geq 4B,$ reverse triangle inequality yields $|E_\theta[\wh \theta]-E_{\theta'}[\wh \theta]|\geq \tfrac 12 |\theta-\theta'|$ and \eqref{eq.1} becomes
\begin{align*}
    \frac{\frac 14 ( \theta-\theta')^2}{4-2r^2(\theta,\theta')}
    \Big( \frac{1}{r(\theta,\theta')}-r(\theta,\theta')\Big)^2
    &\leq 2\sup_{\theta \in \Theta}\Var_\theta\big(\wh \theta\big).
\end{align*}
The left hand side of this inequality does not depend on the estimator $\wh \theta$ anymore. Therefore, this inequality provides us with a lower bound on the worst-case variance of an arbitrary estimator.

While this applies to a one-dimensional parameter, the same procedure can immediately be extended to derive lower bounds on the worst-case variance for pointwise estimation of a function value $f(x_0)$ in a nonparametric statistical model with unknown regression function $f.$ As shown in \cite{2020arXiv200600278D}, one can extend these ideas moreover to derive lower bounds for the integrated variance and for (high-dimensional) parameter vectors. 

Rephrasing the argument leads moreover to lower bounds on the worst-case bias given an upper bound for the worst-case variance. Taking a suitable asymptotics $\theta'\to \theta$ and imposing standard regularity conditions, it can be shown moreover that \eqref{eq.1} converges to the Cram\'er-Rao lower bound (Theorem A.4 in \cite{2020arXiv200600278D}).

\section{Lower bounds for bias-MAD trade-off}

To measure the stochastic error of an estimator, a competitor of the variance is the mean absolute deviation (MAD). For a random variable $X,$ the MAD is defined as $E\big[ |X-u| \big],$ where the centering point $u$ is either the mean or the median of $X$. If centered at the mean, the MAD is upper bounded by $\sqrt{\Var(X)},$ but compared to the variance, less weight is given to large outcomes of $X.$ For a statistical model $(P_\theta:\theta\in \Theta),$ the most natural extension seems therefore to study the trade-off between $m(\theta)-\theta$ and $E_\theta[|\wh \theta-m(\theta)|],$ where again $m(\theta)$ is either the mean or the median of the estimator $\wh \theta$ under $P_\theta$.

The first result provides an abstract inequality that can be used to relate $m(\theta)-\theta$ and $E_\theta[|\wh \theta-m(\theta)|],$ for any centering $m(\theta).$
It can be viewed as an analogue of \eqref{eq.lem1_2}.

\begin{lemma}
\label{lem.abs_deviation}
Let $P,Q$ be two probability distributions on the  same measurable space and write $E_P, E_Q$ for the expectations with respect to $P$ and $Q$. Then for any random variable $X$ and any real numbers $u$, $v$, we have
    \begin{align}
        \frac 15 \big(1 - H^2(P,Q)\big)^2 |u-v|
        \leq E_P \big[\big|X - u \big|\big]
        \vee E_Q\big[\big|X - v \big|\big],
        \label{eq.MAD_lb}
    \end{align}
\end{lemma}

\begin{proof}
Applying the triangle inequality and the Cauchy-Schwarz inequality, we have
\begin{align}
    \begin{split}
    &\big(1- H^2(P,Q) \big) \big| u - v  \big| \\
    &= \int \big| X(\om) - u - X(\om) + v \big|
    \sqrt{p(\om) q(\om)} \, d\nu(\om) \\
    &\leq \int \big| X(\om) - u \big| \sqrt{p(\om) q(\om)} \, d\nu(\om)
    + \int \big| X(\om) - v \big| \sqrt{p(\om) q(\om)} \, d\nu(\om) \\
    &\leq \sqrt{E_P \big[ \big|X - u \big| \big]
    E_Q \big[ \big|X - u \big| \big]}
    + \sqrt{E_P \big[ \big|X - v \big| \big]
    E_Q \big[ \big|X - v \big| \big]}.
    \end{split}
    \label{eq.abs_deviation_basic_ineq}
\end{align}
Bound $E_Q[|X-u|]\leq E_Q[|X-v|]+|u-v|$ and $E_P[|X-v|]\leq E_P[|X-u|]+|u-v|.$ With $a:=E_P[|X-v|] \vee E_Q[|X-u|],$ $b:=|u-v|$ and $d:=1- H^2(P,Q),$ we then have $db\leq 2\sqrt{a^2+ab}$ or equivalently $a^2+ab-d^2b^2/4 \geq 0.$ Since $a\geq 0,$ solving the quadratic equation $a^2+ab-d^2b^2/4=0$ in $a$ gives that $a \geq b (\sqrt{1+d^2}-1)/2.$ Since $0\leq d\leq 1,$ we also have that $\sqrt{1+d^2}-1 \geq 2d^2/5,$ which can be verified by adding one to both sides and squaring. Combining the last two inequalities gives finally the desired result $a\geq bd^2/5.$
\end{proof}

The derived inequality does not directly follow from the triangle inequality $|u-v|\leq |x-u|+|x-v|$ as the expectations on the right-hand side of \eqref{eq.abs_deviation_basic_ineq} are taken with respect to different measures $P$ and $Q.$
Equality up to a constant multiple is attained if $H(P,Q)<1$ and $X=v$ with probability $1$.

An important special case of the previously derived inequality is 
\begin{align}
    \frac 15 \big(1 - H^2(P,Q)\big)^2 \big|E_P[X]-E_Q[X]\big|
        \leq E_P \big[\big|X - E_P[X] \big|\big]
        \vee E_Q\big[\big|X - E_Q[X] \big|\big].
        \label{eq.bsidbb}
\end{align}
Let us now compare this to the change of expectation inequalities involving the variance in the regime where the measures $P$ and $Q$ are close. As mentioned above, $E_P[|X-E_P[X]|]\leq \sqrt{\Var_P(X)}.$ Moreover, $E_P[|X-E_P[X]|]$ and $\sqrt{\Var P(X)}$ are typically of the same magnitude. The Hellinger lower bound for the variance \eqref{eq.lem1_2} is $$\frac 1{\sqrt{4-2H^2(P,Q)}} (1-H^2(P,Q)) \frac{|E_P[X]-E_Q[X]|}{H(P,Q)} \leq \sqrt{\Var_P(X)+ \Var_Q(X)}.$$ Compared to \eqref{eq.bsidbb}, the variance lower bound also includes a term $H(P,Q)^{-1}$ on the left hand side that improves the inequality if the distributions $P$ and $Q$ are close. The next result shows that improving in this regime the inequality \eqref{eq.bsidbb} requires that the likelihood ratio is uniformly close to one. This is much stronger. For instance if $P_\theta$ denotes the distribution of $\mathcal{N}(\theta,1),$ then, $H(P_\theta,P_{\theta'})=1-e^{-\tfrac 18 (\theta-\theta')^2},$ and $H(P_\theta,P_{\theta'})\to 0$ if $\theta -\theta'\to 0.$ However, the likelihood ratio $dP_\theta/dP_{\theta'}$ is unbounded whenever $\theta\neq \theta'.$

\begin{lemma}
Define $0/0$ as $0.$ If $p$ and $q$ are the respective $\nu$-densities of $P$ and $Q,$ then
\begin{align}
    \frac{1 - H^2(P,Q)}{\|\frac{p-q}{p\wedge q}\|_{L^\infty}} \big|E_P[X]-E_Q[X]\big|
        \leq E_P \big[\big|X - E_P[X] \big|\big]
        \vee E_Q\big[\big|X - E_Q[X] \big|\big].
\end{align}
Moreover, if $P,Q$ are defined on a finite probability space, then there exists a random variable $X^*$, such that 
\begin{align}
    E_P \big[\big|X^* - E_P[X^*] \big|\big]
        \vee E_Q\big[\big|X^* - E_Q[X^*] \big|\big]
    \leq \frac{1}{\|\frac{p-q}{p\vee q}\|_{L^\infty}} \big|E_P[X^*]-E_Q[X^*]\big|.    
\end{align}
\end{lemma}

\begin{proof}
Using that $\int (X(\omega)-E_P[X]) p(\omega)\, d\nu (\omega)=0$ and $\int (X(\omega)-E_Q[X]) q(\omega)\, d\nu (\omega)=0,$ we have the identity
\begin{align*}
    &\int \sqrt{p(\omega)q(\omega)} \, d\nu(\omega) \,
    \big(E_P[X]-E_Q[X]\big) \\
    &=
    \int \big(X(\omega) -E_P[X]\big)
    \Big( p(\omega) -\sqrt{p(\omega)q(\omega)}\Big)
    d\nu(\omega)\\ 
    &\quad +
    \int \big(X(\omega) -E_Q[X]\big)
    \Big(\sqrt{p(\omega)q(\omega)}- q(\omega)\Big)
    d\nu(\omega).
\end{align*}
Taking the absolute value inside the integrals gives
\begin{align*}
    &\int \sqrt{p(\omega)q(\omega)} \, d\nu(\omega) \,
    \big|E_P[X]-E_Q[X]\big| \\
    &\leq
    \int \big|X(\omega) -E_P[X]\big| p(\omega) d\nu(\omega) \, 
    \bigg\| 1 -\sqrt{\frac{q(\omega)}{p(\omega)}}\bigg\|_{L^\infty} 
    \\ 
    &\quad +
    \int \big|X(\omega) -E_Q[X]\big| q(\omega) d\nu(\omega) \, 
    \bigg\| \sqrt{\frac{p(\omega)}{q(\omega)}}-1\bigg\|_{L^\infty}.
\end{align*}
By definition of the Hellinger distance, $\int \sqrt{p(\omega)q(\omega)}=1-H^2(P,Q).$ Moreover, for $a,b\geq 0,$ we have $|1-\sqrt{a/b}|=|(b-a)/((\sqrt{a}+\sqrt{b})\sqrt{b})|\leq |(b-a)/(b\wedge a)|.$ Combining these arguments gives
\begin{align*}
        \big(1 - H^2(P,Q)\big)\big|E_P[X]-E_Q[X]\big|
        \leq \Big(E_P \big[\big|X - E_P[X] \big|\big]
        \vee E_Q\big[\big|X - E_Q[X] \big|\big]\Big) \bigg\|\frac{p-q}{p\wedge q}\bigg\|_{L^\infty},
\end{align*}
proving the first claim. 

For the second claim, recall that the probability space $\Omega=\{\omega_j, j=1,2,\ldots, M\}$ is assumed to be finite and denote by $p_j:=P(\{\omega_j\}),$ $q_j:=Q(\{\omega_j\}),$ $j=1,2,\ldots, M$ the respective probability mass functions of $P$ and $Q.$ Define the random variable $X_j(\omega):=\mathbf{1}(\omega=\omega_j).$ Then, $E_P[X_j]-E_Q[X_j]=p_j-q_j,$ $E_P[|X_j-E_P[X_j]|]=p_j(1-p_j),$ and $E_Q[|X_j-E_Q[X_j]|]=q_j(1-q_j).$ For $j^*\in \operatorname{argmax}_{\ell=1,\ldots,M} |p_\ell-q_\ell|/(p_\ell\vee q_\ell),$ \[E_P[|X_{j^*}-E_P[X_{j^*}]|] \vee E_Q[|X_{j^*}-E_Q[X_{j^*}]|]\leq p_{j^*}\vee q_{j^*}= \frac{1}{\max_\ell |\frac{p_\ell-q_\ell}{p_\ell\vee q_\ell}|} \underbrace{|p_{j^*}-q_{j^*}|}_{=|E_P[X_{j^*}]-E_Q[X_{j^*}]|}.\]
\end{proof}

For the application to statistics, the random variable $X$ is an estimator. Thus, given $X$, a related question is to find a random variable $X'$ with $E_P[X']=E_P[X]$ and $E_Q[X']=E_Q[X],$ but smaller mean absolute deviations $E_P[|X'-E_P[X']|]<E_P[|X-E_P[X]|]$ and $E_Q[|X'-E_Q[X']|]<E_Q[|X-E_Q[X]|].$ In particular, for the trade-off between bias and mean absolute deviation, it does not seem favorable that $X$ attains large values, as this mainly increases the mean absolute deviation. If for a measurable set $A,$ the conditional means are the same, that is, $E_P[X| X\in A]=E_Q[X|X\in A],$ then, $X'=X\mathbf{1}(X\in A^c)+E_P[X| X\in A] \mathbf{1}(X\in A)$ satisfies $E_P[X']=E_P[X],$ $E_Q[X']=E_Q[X],$ 
\begin{align*}
    &E_P[|X'-E_P[X']|] \\
    &=E_P[|X'-E_P[X]|] \\
    &=E_P\big[|X-E_P[X]| \, \big| \, X\in A^c\big]P(A^c)+E_P\big[|E_P[X| X\in A]-E_P[X]| \, \big| \, X\in A\big]P(A) \\
    &=E_P\big[|X-E_P[X]| \, \big| \, X\in A^c\big]P(A^c)+\big|E_P[X| X\in A]-E_P[X]\big|P(A) \\
    &=E_P\big[|X-E_P[X]| \, \big| \, X\in A^c\big]P(A^c)+\big|E_P\big[X-E_P[X] \, \big| X\in A\big]\big|P(A) \\
    &\leq E_P\big[|X-E_P[X]| \, \big| \, X\in A^c\big]P(A^c)+E_P\big[\big|X-E_P[X]\big| \, \big| X\in A\big] P(A) \\
    &= E_P[|X-E_P[X]|],
\end{align*}
and similarly $E_Q[|X'-E_Q[X']|]\leq E_Q[|X-E_Q[X]|].$ The argument can be viewed as a variation of the convex loss version of the Rao-Blackwell theorem.

\section{Application to pointwise estimation in the Gaussian white noise model}

In the Gaussian white noise model, we observe a random function $Y=(Y_x)_{x\in [0,1]},$ with
\begin{equation}
    dY_x = f(x) \, dx + \frac{1}{\sqrt{n}} \, dW_x,
    \label{eq.mod_GWN}
\end{equation}
where $W$ is an unobserved standard Brownian motion. The aim is to recover the unobserved, real-valued regression function $f\in L^2([0,1])$ from the data $Y$. Below, we study the bias-MAD trade-off for estimation of $f(x_0)$ with fixed $x_0\in [0,1].$

\medskip

Concerning upper bounds for the MAD risk in this setting, optimal convergence rates are obtained in \cite{MR855002} and the first order asymptotics of the mean absolute deviation risk for Lipschitz functions is derived in \cite{MR1292544}.

To obtain lower bounds for the bias-MAD trade-off, denote by $P_f$ the data distribution of the Gaussian white noise model with regression function $f.$ It is known that the Hellinger distance is
\begin{align}
    H^2(P_f,P_g)=1-\exp\Big(-\frac n8 \|f-g\|_2^2\Big),
    \label{eq.Hell_GWN}
\end{align}
whenever $f,g \in L^2([0,1]),$ see \cite{2020arXiv200600278D} for a reference and a derivation. This means that the inequality \eqref{eq.MAD_lb} becomes 
\begin{align}
    \frac 15 \exp\Big(-\frac n4 \|f-g\|_2^2\Big) |u-v|
        \leq E_f \big[\big|X - u \big|\big]
        \vee E_g\big[\big|X - v \big|\big].
        \label{eq.Hell_GWN_rewritten}
\end{align}
As commonly done in nonparametric statistics, we impose an H\"older smoothness condition on the regression function $f.$ 
Let $R > 0$, $\beta > 0$ and denote by $\floorbeta$ the largest integer that is strictly smaller than $\beta$. On a domain $D \subseteq \Rb,$ we define the $\beta$-H\"older norm by $\| f \|_{\Cc^\beta(D)}= \sum_{\ell \leq \floorbeta}  \|f^{(\ell)} \|_{L^\infty(D)}+ \sup_{x, y \in D, x\neq y} |f^{(\floorbeta)}(x) - f^{(\floorbeta)}(y)  | / | x - y  |^{\beta - \floorbeta},$ with $L^\infty(D)$ the supremum norm on $D$ and $f^{(\ell)}$ denoting the $\ell$-th (strong) derivative of $f$ for $\ell \leq \floorbeta$. For $D=[0,1],$ let $\Cc^\beta(R):=\{f:[0,1] \to \Rb :\| f \|_{\Cc^\beta([0,1])} \leq R\}$ be the ball of $\beta$-H\"older smooth functions $f:[0,1] \to \Rb$ with radius $R.$ We also write $\Cc^\beta(\Rb):=\{K:\Rb \to \Rb :\| K \|_{\Cc^\beta(\Rb)} < \infty\}.$

\begin{thm}\label{thm.MAD_lb}
Consider the Gaussian white noise model \eqref{eq.mod_GWN} with parameter space $\Cc^\beta(R)$. Let $C>0$ be a positive constant. If $\wh f(x_0)$ is an estimator for $f(x_0)$ satisfying
\begin{equation*}
    \sup_{f\in \Cc^\beta(R)} \big|\Bias_f(\wh f(x_0))\big|
    < \Big(\frac C n\Big)^{\beta/(2\beta+1)},
\end{equation*}
then, there exist positive constants $c=c(C,R)$ and $N=N(C,R),$ such that
\begin{align*}
        \sup_{f\in \Cc^\beta(R)} E_f \big[\big|\wh f(x_0) -  E_f[\wh f(x_0)] \big|\big]
        &\geq c n^{- \beta / (2 \beta + 1)}, \quad \text{for all} \ n\geq N.
\end{align*}
Explicit expressions for $c$ and $N$ can be derived from the proof. If $\Med_f[\wh f(x_0)] |]$ denotes the median of $\wh f(x_0),$ then the same holds if $\Bias_f(\wh f(x_0))$ and $E_f[|\wh f(x_0) -  E_f[\wh f(x_0)] |]$ are replaced by $\Med_f[\wh f(x_0)]-f(x_0)$ and $E_f[|\wh f(x_0) -  \Med_f[\wh f(x_0)] |],$ respectively.
\end{thm}

The result is considerably weaker than the earlier derived lower bounds for the bias-variance trade-off for pointwise estimation. This is due to the fact that \eqref{eq.MAD_lb} is less sharp. Nevertheless, the conclusion provides still more information than the minimax lower bound for the absolute value loss. To see this, observe that by the triangle inequality,

\begin{align*}
  \sup_{f\in \Cc^\beta(R)} E_f [|\wh f(x_0) -  E_f[\wh f(x_0)]|]
  &\geq \sup_{f\in \Cc^\beta(R)} E_f[|\wh f(x_0)-f(x_0)|]-|\Bias_f(\wh f(x_0))|  \\
  &\geq \sup_{f\in \Cc^\beta(R)} E_f[|\wh f(x_0)-f(x_0)|]-\sup_{f\in \Cc^\beta(R)} |\Bias_f(\wh f(x_0))|.
\end{align*}
Thus, the conclusion of Theorem \ref{thm.MAD_lb} can be deduced from the minimax lower bound $\sup_{f\in \Cc^\beta(R)} \linebreak E_f[|\wh f(x_0)-f(x_0)] \geq (K/n)^{\beta / (2 \beta + 1)},$ as long as $C<K.$ Arguing via the minimax rate, nothing, however, can be said if $C>K,$ that is, the bias is of the optimal order with a potentially large constant. Indeed, if we would change the role of the worst-case bias and the worst-case risk in the previous display, we get the lower bound
\begin{align*}
  \sup_{f\in \Cc^\beta(R)} E_f [|\wh f(x_0) -  E_f[\wh f(x_0)]|]
  &\geq \sup_{f\in \Cc^\beta(R)} |\Bias_f(\wh f(x_0))|-\sup_{f\in \Cc^\beta(R)} E_f[|\wh f(x_0)-f(x_0)|].
\end{align*}
Since $\wh f(x_0)$ is an arbitrary estimator, we cannot exclude the possibility that \[\sup_{f\in \Cc^\beta(R)} \big|\Bias_f(\wh f(x_0))\big|\newline \approx \sup_{f\in \Cc^\beta(R)} E_f\big[|\wh f(x_0)-f(x_0)|\big].\] However, Theorem \ref{thm.MAD_lb} shows that even in the case $C>K$, the worst-case variance cannot converge faster than $n^{-\beta/(2\beta+1)}.$

\begin{proof}[Proof of Theorem \ref{thm.MAD_lb}]
For any function $K \in \Cc^\beta(\Rb)$ satisfying $K(0) = 1$ and $\|K\|_2 < +\infty,$ define $V:=R/\|K\|_{\Cc^\beta(\Rb)},$ $r_n:=(2/V)^{1/\beta}(C/n)^{1/(2\beta+1)},$ and
\begin{equation*}
    \Fc:=\Big\{f_\theta(x) =
    \theta
    V r_n^\beta K \Big( \frac{ x - x_0 }{r_n} \Big) : |\theta| \leq  1 \Big\}.
\end{equation*}
By Lemma B.1 in \cite{2020arXiv200600278D}, we have for $0 < h \leq 1,$ $\|h^\beta K ( (\cdot - x_0 )/h )\|_{\Cc^\beta(\Rb)}\leq \|K\|_{\Cc^\beta(\Rb)}.$ Since $r_n\leq 1$ for all sufficiently large $n$, taking $h=r_n,$ we find $\|f_\theta\|_{\Cc^\beta([0,1])}\leq |\theta| V \|K\|_{\Cc^\beta(\Rb)} \leq R$ for all $\theta \in [-1,1].$ Thus, $\Fc \subseteq \Cc^\beta(R)$ whenever $r_n \leq 1.$ We can now apply \eqref{eq.Hell_GWN_rewritten} to the random variable $\wh f(x_0)$ choosing $P=P_{f_{\pm 1}},$ $Q=P_0$ and centering $u=E_{f_{\pm 1}}[\wh f(x_0)],$ $v=E_0[\wh f(x_0)]$,
\begin{align*}
    &\frac 15 \exp\Big(- \frac{n}4 \|f_{\pm 1}\|_2^2\Big)
    \Big| E_{f_{\pm 1}}[\wh f(x_0)] - E_0[\wh f(x_0)]\Big| \\
    &\leq E_{f_{\pm 1}} \big[\big|\wh f(x_0) -  E_{f_{\pm 1}} [\wh f(x_0)] \big|\big]
    \vee E_0\big[\big|\wh f(x_0) - E_0[\wh f(x_0)] \big|\big].
\end{align*}
Using substitution and the definition $r_n=(2/V)^{1/\beta}(C/n)^{1/(2\beta+1)},$ we find $$\|f_{\pm 1}\|_2^2\leq V^2 r_n^{2\beta} \int_{\Rb} K^2\Big(\frac{x-x_0}{r_n}\Big) \, dx=
V^2 r_n^{2\beta+1}\|K\|_2^2= \frac 1n 2^{2+1/\beta} V^{-1/\beta} C\|K\|_2^2$$ and so,
\begin{align*}
    \frac 15 \exp\bigg( -\Big(\frac 2V\Big)^{1/\beta} C\|K\|_2^2\bigg) \big| E_{f_{\pm 1}}[\wh f(x_0)] - E_0[\wh f(x_0)]  \big|
    &\leq 
    \sup_{f \in \Cc^\beta(R)} E_f \big|\wh f(x_0) -  E_f[\wh f(x_0)] \big|.
\end{align*}
Due to $K(0)=1$ and the definition of $r_n,$ we have $f_{\pm 1}(x_0)= \pm V r_n^\beta= \pm 2(C/n)^{\beta/(2\beta+1)}$ and because of the bound on the bias, $E_{f_1}[\wh f(x_0)]\geq (C/n)^{\beta/(2\beta+1)}$ and $E_{f_{-1}}[\wh f(x_0)]\leq -(C/n)^{\beta/(2\beta+1)}.$ Choosing for the lower bound $f_1$ if $E_{f_0}[\wh f(x_0)]$ is negative and $f_{-1}$ if $E_{f_0}[\wh f(x_0)]$ is positive, we find
\begin{align*}
    \frac 15 \exp\bigg( -\Big(\frac 2V\Big)^{1/\beta} C\|K\|_2^2\bigg) \Big(\frac{C}{n}\Big)^{\frac{\beta}{2\beta+1}}
    \leq
    \sup_{f \in \Cc^\beta(R)} E_f \big|\wh f(x_0) -  E_f[\wh f(x_0)] \big|,
\end{align*}
proving the claim. The proof for the median centering follows exactly the same steps. 
\end{proof}

\section{Further extensions of the bias-variance trade-off}

A natural follow-up question is to wonder about other concepts to measure systematic and stochastic error of an estimator. This section is intended as an overview of related concepts.

A large chunk of literature on variations of the bias-variance trade-off is concerned with extensions to classification under $0$\,-\,$1$ loss, see \cite{kohavi1996bias, breiman1996bias, tibshirani1996bias, james1997generalizations}. These approaches have been compared
in~\cite{rozmusmethods}. \cite{le2005bias} proposes an extension to the multi-class setting.
In a Bayesian framework, \cite{wolpert1997bias} argues that the bias-variance trade-off becomes a bias-covariance-covariance trade-off, where a covariance correction is added. For relational domains, \cite{neville2007bias} propose to separate the bias and the variance due to the learning process from the bias and the variance due to the inference process. Bias-variance decompositions for the Kullback-Leibler divergence and for the log-likelihood are studied in~\cite{Heskes1998likelihood_based}.
Somehow related, \cite{wu2012decomposition} introduces the Kullback-Leibler bias and the Kullback-Leibler variance, and shows, using information theory, that a similar decomposition is valid.
\cite{Domingos2000unified} propose generalized definitions of bias and variance for a general loss, but without showing a bias-variance decomposition. For several exponential families \cite{hansen2000general} shows that there exist a loss $L$ such that a bias-variance decomposition of $L$ is possible.
\cite{James2003variance} studied a bias-variance decomposition for arbitrary loss functions, comparing different ways of defining the bias and the variance in such cases.

\section*{Acknowledgements}
The project has received funding from the Dutch Research Council (NWO) via the Vidi grant VI.Vidi.192.021.



\bibliographystyle{elsarticle-harv} 
\bibliography{biblio}

\end{document}